\documentclass[]{amsart}

\usepackage{amsthm}
\usepackage{amssymb}
\usepackage{enumitem}
\usepackage{easybmat}
\usepackage{multirow,bigdelim}

\title{Fixed Points and Components of equivalued affine Springer fibers}
\author{Pablo Boixeda Alvarez}

\newtheorem{thm}{Theorem}[section]
\newtheorem{alg}{Algorithm}
\newtheorem{lem}{Lemma}[section]
\newtheorem{prop}{Proposition}[section]

\theoremstyle{remark}
\newtheorem{rmk}{Remark}[section]
\theoremstyle{definition}
\newtheorem{definition}{Definition}[section]
\newtheorem{case}{Case}
\begin{document}

\begin{abstract}
 For G a semisimple algebraic group, we revisit the description of the components of the affine Springer fiber given by $ts$, with $s$ a regular semisimple element. We then compute the fixed points of each component of a particular affine Springer fiber for Type A 
\end{abstract}
\maketitle
\section{Introduction}

Let $G$ be a semisimple group, $B$ a Borel and $N$ the unipotent radical. Denote by $\mathcal{K}=\mathbb{C}((t))$ and $\mathcal{O}$ and let the set of Iwahori subalgebras $\mathcal{F}l$ the affine flag variety. This is an ind-scheme, with points over $\mathbb{C}$ given by $G(\mathcal{K})/\mathfrak{B}$, where $\mathfrak{B}\subset G(\mathcal{O})$ is the Iwahori subgroup lifting $B$ under the map $t\mapsto 0$. We denote by $W$ the affine Weyl group and $W_f$ the finite Weyl group, associated to this group. Consider the root system for $G$, given by$(\mathbb{X}=Hom(T,\mathbb{G}_m),R,\mathbb{X}^\vee =Hom(\mathbb{G}_m,T),R^\vee)$, where $T$ is a maximal torus and $R$ (resp $R^\vee$) are the set of roots (resp coroots) of the group. Then we have $W=W_f\ltimes \mathbb{Z}R^\vee$ and ${}^fW$ the minimal length elements in the cosets of $W_f\char`\\ W$. We  consider the affine roots $R\times\mathbb{Z}\delta$. The choice of $B$ determines a set of positive roots $R^+$ and the choice of Iwahori $\mathfrak{B}$ corresponding to $B$ gives the choice of positive affine roots given by $R^+\times \{0\}\amalg R\times \mathbb{Z}_{>0}\delta$. \\
Further recall that $W$ acts on $\mathbb{X}\otimes_{\mathbb{Z}} \mathbb{R}$ by affine linear transformations, generated by affine reflections on the planes $H_{\alpha,n}=\{\lambda\in \mathbb{X}\otimes_{\mathbb{Z}} \mathbb{R}|<\lambda,\alpha>=n\}$ given by $\alpha\in R$ $n\in\mathbb{Z}$. Consider the connected components $\mathbb{X}\otimes_{\mathbb{Z}} \mathbb{R}\setminus\bigcup_{\alpha,n}H_{\alpha,n}$. Recall the closure of these are called alcoves. We denote the fundamental alcove by $A_0=\{\lambda|0\leq<\lambda,\alpha>\leq 1\forall\alpha R^+\}$. With this we get a bijection between $W$ and the set of alcoves, given by $w\mapsto wA_0$. We will use this bijection and refer to $wA_0$ as the alcove corresponding to $w\in W$.\\
We now introduce a particular affine Springer fiber. Let $s\in\mathfrak{g}$ a regular semisimple element of the Lie algebra. We will consider the affine Springer fiber for the element of the loop Lie algebra $\gamma=ts$, ie $\mathcal{F}l_\gamma=\{\mathfrak{b}\in\mathcal{F}l|\gamma\in\mathfrak{b} \}$.\\
It is known that this space is equidimensional of dimension $dim(G/B)$ of the finite flag variety. Note that this space has an action of $T(\mathcal{K})$ as this comutes with $ts$, so we get an action of the coroot lattice $\mathbb{X}^\vee$ and of the torus $T$. Further there is an action of $\mathbb{G}_m$ on $\mathcal{F}l$ given by rescaling the parameter $t$ of $\mathcal{K}$. This preserves $\mathcal{F}l_\gamma$ as this acts on $ts$ by scaling and hence fixes the corresponding affine Springer fiber. Denote by $\widetilde{T}=T\times\mathbb{G}_m$. The fixed points of $T$ and of $\widetilde{T}$ are both the same and are given by $W_{ext}=W_f\ltimes \mathbb{X}^\vee$ under the map $wt^\lambda\mapsto\dot{w}t^\lambda\mathfrak{B}/\mathfrak{B}$, for $\lambda$ a coweight, with $t^\lambda$ given by the image of $t$ under $\lambda:\mathcal{K}^*\mapsto T(\mathcal{K})$, and $\dot{w}$ a lift of the finite Weyl group element.\\
The goal of this paper is to understand the fixed points on each component. Due to results from Goresky Kottwitz and MacPherson\cite{GKM2}, understanding the fixed points and 1 dimensional orbits would give a combinatorial description of the cohomology of each component.\\ 
Understanding the fixed points of a component is a problem that remains open for finite Springer fibers and is thus a special property of this affine Springer fiber that such a result can be computed.\\
In work by Fresse\cite{F}, he determines the smoothness properties for particular finite Springer fibers. To do this he studies properties of torus fixed points. In a similar way understanding the one dimensional orbits and fixed points would give an understanding of some smothness properties for these components.\\
Further the motivation for this result was a question by Le Hung on the degree of some bundles on restriction to the components. This questions comes motivated by the study of Shimura varieties and the results computed can be used in that study.\\
 The geometry and cohomology of these affine Springer fibers have also been considered before. In particular, work of Hikita\cite{H} and Kivinen\cite{K} expand on the combinatorial description given in \cite{GKM2} and relate the results to the Hilbert scheme of points on $\mathbb{C}^2$ for the case of Type A. Further the geometry of this affine Springer fiber has also been studied by Lusztig in \cite{L}.\\
 In upcoming work by Bezrukavnikov and McBreen \cite{BM} an equivalence between microlocal sheaves on this affine Springer fiber and a regular block of the category of graded representation of the small quantum group is proven. Also in upcoming work with Bezrukavnikov, Shan and Vasserot \cite{BBASV} the center of the Lusztig quantum group lying in the small quantum group is described as the invariants of the cohomology of this affine Springer fiber under the root lattice $\mathbb{Z}R$. This is conjecturally the whole center of the small quantum group in Type A. These 2 results give a strong connection betweeen the geometry of this space and the representation theory of the small quantum group. This further is related to the characteristic p representation theory of reductive algebraic groups, by work of AJS. In particular the above results could lead to a combinatorial description of the endomorphisms of projective objects in these categories.\\
 \subsection{Organization of the paper} In section 2 we recall a particular description of the components of $\mathcal{F}l_\gamma$ following \cite{GKM1}. In section 3 we describe some attracting neighborhoods of the fixed points and state the result. In appendix A we state and prove some combinatorial lemmas that we will need. In appendix B we carry out the combinatorial computations giving a proof of the results.
 \subsection{Acknowledgments} I want to thank Viet Bao Le Hung for asking the question that motivated this work, as well as for very fruitful discussions about the project. I also want to thank Roman Bezrukavnikov for a lot of conversations about the topic and for giving me great suggestions to improve all my work.
\section{Components of equivalued affine Springer fibers}
We first start by giving some description of the components. The usual description gives a bijection between $W_{ext}$ and the components given by the assignments $w\mapsto \overline{N(\mathcal{K})w\mathfrak{B}/\mathfrak{B}}$.\\ 
To compute the fixed points of components we will want a different description of the components. We do this by understanding intersection of the space with $G(\mathcal{O})$-orbits.\\
To understand the intersections of $G(\mathcal{O})$-orbits with $\mathcal{F}l_\gamma$ we follow \cite{GKM1}. We will do the computations for a simply connected group to make the notation easier. Before starting this description, recall that $G(\mathcal{O})$-orbits in $\mathcal{F}l$ can be given by $w\in {}^fW$, so from now on we will assume this condition. To continue note that by Lemma A.1 we have that ${}^w\mathfrak{B}\cap \mathfrak{B}= G(\mathcal{O})\cap {}^w\mathfrak{B}$, ie ${}^w\mathfrak{B}$ contains all finite positive roots and hence no finite negative roots.\\
Now to understand the intersection we see that $G(\mathcal{O})w\mathfrak{B}/\mathfrak{B}$ is isomorphic to $X=G(\mathcal{O})/{}^w\mathfrak{B}\cap\mathfrak{B}$. Further we can consider the intersection with $\mathcal{F}l_\gamma$ as $\{g\in G(\mathcal{O})/{}^w\mathfrak{B}\cap\mathfrak{B}|ts\in {}^{gw}\mathfrak{b}\}$.\\
We rewrite this as the subvariety $Y\subset X$ given by the vanishing of a map of vector bundles as follows:\\
Consider the map $\mathfrak{g}(\mathcal{O})\rightarrow t\mathfrak{g}(\mathcal{O})$ given by bracketing with $ts$. This induces a map of bundles
$$[ts,-]:\mathfrak{g}(\mathcal{O})/{}^{gw}\mathfrak{b}\cap\mathfrak{g}(\mathcal{O})\rightarrow t\mathfrak{g}(\mathcal{O})/{}^{gw}\mathfrak{b}\cap t\mathfrak{g}(\mathcal{O})$$
Note that this map is surjective, as $s$ is regular semisimple. Thus we get that the intersection is a smooth variety, as then the tangent space of the subvariety $Y$, which can be computed as the kernel of this map, is of constant dimension.\\
Further we can consider the following maps
$$X_r=G(\mathcal{O})/({}^w\mathfrak{B}\cap\mathfrak{B})G_r\rightarrow X_{r-1}=G(\mathcal{O})/({}^w\mathfrak{B}\cap\mathfrak{B})G_{r-1}$$
Here $G_r$ are the unipotent subgroup of $G(\mathcal{O})$ with Lie algebra $t^r\mathfrak{g}(\mathcal{O})$ for $r>0$.\\
Similarly we can consider $Y_r\subset X_r$ as given by the vanishing of the map of vector bundles
$$[ts,-]:\mathfrak{g}(\mathcal{O})/{}^{gw}\mathfrak{b}\cap\mathfrak{g}(\mathcal{O})+t^r\mathfrak{g}(\mathcal{O})\rightarrow t\mathfrak{g}(\mathcal{O})/{}^{gw}\mathfrak{b}\cap t\mathfrak{g}(\mathcal{O})+t^{r+1}\mathfrak{g}(\mathcal{O})$$
Note that acting on an element $x\in t\mathfrak{g}(\mathcal{O})$ by an element of $g_r\in G_r$ the difference $x-g_rx$ is in $t^{r+1}\mathfrak{g}(\mathcal{O})$. And further just as above if we start with a regular semisimple this difference gives a surjective map mod ${}^w\mathfrak{b}\cap\mathfrak{g}(\mathcal{O})$ and hence it follows that the map $Y_r\rightarrow Y_{r-1}$ induced from the above map is surjective.\\
Further it is just given as a torsor over the vector bundle given as the kernel of 
$$[ts,-]:t^{r}\mathfrak{g}(\mathcal{O})/{}^{gw}\mathfrak{b}\cap t^{r}\mathfrak{g}(\mathcal{O})+t^{r+1}\mathfrak{g}(\mathcal{O})\rightarrow t^{r+1}\mathfrak{g}(\mathcal{O})/{}^{gw}\mathfrak{b}\cap t^{r+1}\mathfrak{g}(\mathcal{O})+t^{r+2}\mathfrak{g}(\mathcal{O})$$
Using this we can reduce the computation of irreducibility by repeatedly mapping down along those maps to computing the irreducibility of the subvariety  $Y_1$ of $G(\mathcal{O})/({}^w\mathfrak{B}\cap\mathfrak{B})G_1\cong G/B$. To describe $Y_1$ let's introduce some notation. Let $V_w={}^w\mathfrak{b}\cap t\mathfrak{g}(\mathcal{O})/t^2\mathfrak{g}(\mathcal{O})\subset t\mathfrak{g}(\mathcal{O})/t^2\mathfrak{g}(\mathcal{O}) \cong\mathfrak{g}$. Using this, $Y_1$ is given as the subset
$$\{gB\in G/B|s\in {}^{g}V_w\}$$
These varieties are called Hessenberg varieties, which by the above arguments are closed smooth subvarieties of $G/B$ and hence it is projective. Note that this variety is non-empty if $V_w\supset \mathfrak{b}$. In the above case, as $w\in {}^fW$, it is always the case that $V_w\supset \mathfrak{b}$, as by Lemma A.1 all positive finite root subspaces are contained in ${}^w\mathfrak{b}$. Further we see that this variety is $T$ stable (as ${}^w\mathfrak{b}$ is $T$ stable) and it contains all fixed points. So from smoothness it follows that irreducibility is equivalent to connectedness and further connectedness is equivalent to all fixed points being on the same connected component. Note that if the negative simple root spaces are contained in $V_w\subset\mathfrak{g}$ the $\mathbb{P}^1$ connecting $x, xs\in W_f$ for some simple reflection $s$ is contained in $Y_1$. Further if a simple root is missing consider the parabolic $P$ given by the subroot system given by excluding this simple root. Then we have $Y_1\subset \bigcup_{x\in W_f} xP/B$ and hence the space is not connected.\\
This reduces then the condition of irreducibility to the condition that $\delta-\alpha_i$ root space for a simple root $\alpha_i$ is contained in the Lie algebra ${}^w\mathfrak{b}$. Hence we can translate this to the statement that $w^{-1}(\delta-\alpha_i)$ is a positive root. This is equivalent to $w\in F$ the fundamental box, ie these are $w\in W$ corresponding to the alcoves satisfying $0\leq<\lambda,\alpha_i>\leq 1$ for all finite simple roots $\alpha_i$ and all $\lambda$ in the alcove. It follows that we have the following Proposition
\begin{prop}
	The intersection $G(\mathcal{O})w\mathfrak{B}/\mathfrak{B}\cap\mathcal{F}l_\gamma$ is irreducible of dimension $dim(G/B)=\sharp R^+$ for $w\in F$. We denote the closure of this intersection by $Y_w$
\end{prop} 
\begin{proof}
	Note from the above it follows the irreducibility, so it remains to show that these intersections are of the correct dimension. Note that by the computation of the tangent space the dimension is given by $$dim(\mathfrak{g}(\mathcal{O})/{}^{gw}\mathfrak{b}\cap\mathfrak{g}(\mathcal{O}))-dim(t\mathfrak{g}(\mathcal{O})/{}^{gw}\mathfrak{b}\cap t\mathfrak{g}(\mathcal{O}))$$
	which can be rewritten as
	$$dim(t^{-1}{}^w\mathfrak{b}\cap\mathfrak{g}(\mathcal{O})/{}^w\mathfrak{b}\cap\mathfrak{g}(\mathcal{O}))$$
	But for each of the finite positive roots $\alpha$ exactly one of $w^{-1}(\pm\alpha)$ are positive roots, thus exactly half of the finite roots contribute one dimension to $t^{-1}{}^w\mathfrak{b}\cap\mathfrak{g}(\mathcal{O})/{}^w\mathfrak{b}\cap\mathfrak{g}(\mathcal{O})$, thus the dimension is $\sharp R^+$ as required.
\end{proof}
\begin{rmk}
	Note that these are not all the components up to coroot translations, but if we consider the action of the center, we indeed can understand all the components up to translation and central action if we understand the components $Y_w$ for $w\in F$. Hence if we proof results for the above component $Y_w$ we can infer a solution for all components. In this case knowing the fixed points of $Y_w$ for $w\in F$ give the fixed points of all components.
\end{rmk}
\begin{rmk}
	Note further that this already proves the result $Y_w^T\subset \{y\leq w_0w\}$, as $\overline{G(\mathcal{O})w\mathfrak{B}/\mathfrak{B}}$ has fixed points exactly given by $\{y\leq w_0w\}$
\end{rmk}
\section{Fixed points of components}
\subsection{Neighborhoods of fixed points}
In this section we introduce a method of determining weather a fix point is in the closure of a $\widetilde{T}$-stable subset of $\mathcal{F}l$. We then proceed to check that these conditions are satisfied for the components of the affine Springer fiber $\mathcal{F}l_\gamma$ and all fix points smaller in the Bruhat order.\\
Denote by $U_-\subset G(\mathbb{C}[t^{-1}])\subset G(\mathcal{K})$, the preimage of $N_-$ the lower triangular matrices under the map $G(\mathbb{C}[t^{-1}])\rightarrow G$ given by setting $t^{-1}$ to 0.\\
Consider the subset $U_1=U_-\mathfrak{B}/\mathfrak{B}\subset\mathcal{F}l$ is a homogenous space under the action of $U_-$ and at the identity both have the same tangent space, so this gives an open neighborhood of $1$. Note that this group is stable under conjugation by $T$ and further is stable under scaling $t$, so this gives an open $\widetilde{T}$-stable neighborhood of $1$.\\
Further this doesn't contain any other fixed points as if you use a strictly positive cocharacter of $\lambda:\mathbb{G}_m\rightarrow\widetilde{T}$, this set contracts to $1$ under the action of $\mathbb{G}_m$ through $\lambda$ as we take the limit at $0$. Ie this is exactly the attracting set of $1$ for strictly positive cocharacters. If we consider the other $U_-$ orbits we get the attracting sets of all other fix points for a strictly positive cocharacter. In particular this gives all $U_-$ orbits\\
Similarly for $w\in W$ we can consider $U_w=wU_-\mathfrak{B}/\mathfrak{B}\subset\mathcal{F}l$ as a $\widetilde{T}$-stable open neighborhood of $w$ only containing one fix point, and in fact which can be described as the attracting set of $w$ for some cocharacter.\\
Now we have the following easy lemma
\begin{lem}
	A $\widetilde{T}$-stable subset $Y\subset\mathcal{F}l$ satisfies $w\in \overline{Y}$ if and only if $U_w\cap Y$ is non-empty
\end{lem}
\begin{proof}
	\begin{itemize}[leftmargin=*,label={}]
		\item $(\Rightarrow)$: $U_w$ is an open neighborhood of $w$ so this direction is clear
		\item $(\Leftarrow)$: If $U_w\cap Y$ is non-empty then the intersection is $\widetilde{T}$ stable and in the attracting set of $w$ for some cocharacter, so the limit points of this $\mathbb{G}_m$ action are contained in $\overline{Y}$ and hence $w\in\overline{Y}$
	\end{itemize}
\end{proof}
\subsection{Fixed points of components in Type A}
In this section we state our result, which will be proven in Appendix B. The strategy of the proof is given by an explicit combinatorial computation using the ideas from the previous sections
\begin{thm}
	For $w\in F$ the $T$ fixed points in the corresponding component $Y_w$ are given by:
	$$Y_w^T=\{y\leq w_0w\}$$
\end{thm}
\begin{proof}
	Most of the details of the proof are left to the Appendix, but here we sketch the ideas used.\\
	The first idea is giving an explicit description of the intersections described in Section 2, for Type A. This is done in Appendix B2.\\
	Similarly we give an explicit description of the open attractive neighborhoods of each fixed point described in Section 3.1.\\
	Using these description, we can reduce the condition of a particular fixed point being in the closure of the intersection previously described to an explicit determinant not vanishing.\\
	We can reduce the above result to the case $y\in {}^fW$, because multiplying $\mathcal{F}l_\gamma$ by $x\in W_f$ does not fix the affine Springer fiber, but it sends this affine Springer fiber to the one given by $t{}^xs$, which is again of the same form for a different regular semisimple element. Further as we're considering $G(\mathcal{O})$-orbits and $G(\mathcal{O})$ is $W_f$ stable. It follows that the $G(\mathcal{O})$ orbits are preserved. Thus the components $Y_w$ for $w\in W$ are sent to the corresponding components of $\mathcal{F}l_{t{}^xs}$. Hence we see if we can check $y\in W$ is in one of the above given components $Y_w$ for every regular semisimple $s$, we can also show $xy$ is in $Y_w$ for $x in W_f$.\\
	The last part of the argument is a combinatorial strategy to find a monomial in the determinant above mentioned, whose coefficient is easy to compute and can be shown to be non-zero in the case of $y\leq w_0w$ and $y\in{}^fW$. This is done by considering some "high degree" monomials in some sense. To be precise the open part of the component is given as an affine space with a particular $\widetilde{T}$ action. We will consider a variable describing the representation to have high degree if the loop rotation $\mathbb{G}_m$ acts with a high power. We then define the high degree monomials by lexicographic order on the high degree variables. The precise details of these combinatorics are left to the appendix.
\end{proof}
\appendix
\section{Combinatorial lemmas}
In this section we want to state some combinatorial lemmas. We will use these results throughout the paper.\\
\begin{lem}
	For $w\in W$, $w\in {}^fW$ if and only if $w^{-1}(\alpha_i)>0$ for the finite simple roots $\alpha_i$
\end{lem}
\begin{proof}
	Note that $w\in {}^fW$ if and only if $s_iw>w$ for all finite simple reflections, which by the usual argument is equivalent to $w^{-1}(\alpha_i)>0$
\end{proof}
From now on we will assume we're in Type A. Denote by $\omega_i=(1,1,1...1,0,0...0)$ the ith fundamental weight. The fundamental alcove $A_0$ is given in type A by the simplex with vertices given by $\omega_i$. With this we have the following 
\begin{lem}
	If $w\in {}^fW$, then $w\in {}^fW$ if and only if $w(\omega_i)\in \mathbb{X}^{++}\forall i$ and $w(0)\in\mathbb{X}^{++}$. Note that this is equivalent to $wA_0$ being in the dominant chamber.
\end{lem}
\begin{proof}
	For $w\in W$, the condition that it is minimal in its $W_f$ orbit is equivalent to $s_iw\geq w$ for all simple reflections $s_i$, which again is equivalent to $w^{-1}(\alpha_i)$ being a positive root. Write $w=\bar{w}t^\mu$ for $\bar{w}\in W_f$ and $\mu\in \mathbb{Z}R$, then $w^{-1}(\alpha_i)=\bar{w}^{-1}(\alpha_i)+\delta<\mu,\bar{w}^{-1}(\alpha_i))>$, so this is positive if $<\mu,\bar{w}^{-1}(\alpha_i)>\geq0$ and if $<\mu,\bar{w}^{-1}(\alpha_i)>=0$ and $w^{-1}(\alpha_i)$ a positive root.\\
	Now $w(\omega_i)\in \mathbb{X}^{++}$ is equivalent to $0\leq<w(\omega_i),\alpha_j>=<\mu+\omega_i,\bar{w}^{-1}(\alpha_j)>$ for all finite simple roots $\alpha_j$. Now $<\omega_i,\alpha>\geq -1$ $\forall \alpha\in R$ for type A, so we get this is indeed equivalent to the above conditions.\\
\end{proof}
\begin{rmk}
	Note in the above lemma we can change it to work for any type if we use $\frac{1}{\omega_i(\alpha_0)}\omega_i\in\mathbb{X}\otimes_\mathbb{Z}\mathbb{Q}$, where $\alpha_0$ is the largest root. These are exactly the vertices of the fundamental alcove in the other types.
\end{rmk}
Using this characterisation, we will find a different description of the Bruhat order for elements of ${}^fW$. First we introduce an order on $\mathbb{X}^{++}$, which we will also denote by $\leq$, given by $\lambda\leq\mu$ if and only if $\mu-\lambda\in \mathbb{N}R^+$
\begin{lem}
	If $w$, $w'\in{}^fW$, then $w\leq w'$ if and only if $w(\omega_i)\leq w'(\omega_i)$ $\forall i$ and $w(0)\leq w'(0)$
\end{lem}
\begin{proof}
	We first describe the Bruhat order in terms of the alcoves. First we introduce the concept of alcove walks.
	\begin{definition}
		An alcove walk is a sequence of alcoves $(A_1,...A_n)$ such that $A_i$ and $A_{i+1}$ share a codimension one wall. 
	\end{definition}
	With this definition the length of $w\in W$ corresponds to the length of the shortest alcove walk starting at $A_0$ and ending at $wA_0$. This is because the condition that two alcoves $w_1A_0$ and $w_2A_0$ share a codimension 1 wall is equivalent to $w_1=w_2s$ for some simple reflection $s$.\\
	It follows from this description that for a reflection $r$, $wr\leq w$ in the Bruhat order if and only if the reflection hyperplanes lies between $A_0$ and $wA_0$. It is then clear that the difference between the vertices of $wA_0$ and $wrA_0$ is given by a positive scalar of a positive root $\alpha$, corresponding to the reflection hyperplane.\\
	The Bruhat order is generated by the inequality $wr\leq w$ for reflections, so the result follows
\end{proof}
We now introduce a different characterisation of the order on $\mathbb{X}^{++}$ introduced above
\begin{lem}
	If $\lambda$, $\mu\in\mathbb{X}^{++}$ $\lambda\leq\mu$ if and only if $\sum_{j=r}^{\lambda_1}\sharp\{\lambda_i\geq j\}\leq \sum_{j=r}^{\mu_1}\sharp\{\mu_i\geq j\}$ $\forall r$ $(\dagger)$
\end{lem}
\begin{proof}
	First note that $\lambda\leq\mu$ if and only if $\sum_{i=1}^{k}\lambda_i\leq\sum_{i=0}^{k}\mu_i$ $\forall k$ $(\star)$.\\
	Further for $\lambda\in\mathbb{X}^{++}$  $\sum_{j=r}^{\lambda_1}\sharp\{\lambda_i\geq j\}=\sum_{\lambda_i\geq r}\lambda_i-r$ \\
	We proof the equivalence of these 2 sets of inequalities.
	\begin{itemize}[leftmargin=*,label={}]
		\item $(\Leftarrow)$: We will proof $\sum_{i=1}^{k}\lambda_i\leq\sum_{i=1}^{k}\mu_i$ $\forall k$ by induction on $k$. Clearly for $k=0$, both sums are empty so the result is obvious. Assuming the result for $k'< k$ and we proof the inequality for $k$. First if $\lambda_k\leq\mu_k$, then the inequality $(\star)$ for $k$, follows from the same inequality for $k-1$. So we can assume $\lambda_k\geq\mu_k$. Use the given inequality $(\dagger)$ for $r=\mu_k$ to get 
		$$\sum_{i=1}^{k}\lambda_i-\mu_k\leq\sum_{\lambda_i\geq\mu_k}\lambda_i-\mu_k\leq\sum_{\mu_i\geq\mu_k}\mu_i-\mu_k=\sum_{i=1}^{k}\mu_i-\mu_k$$
		Hence the result follows. Here we use that $\lambda_i\leq\lambda_j$ if $i\geq j$ and similarly for $\mu$
		\item $(\Rightarrow)$: To prove $(\dagger)$ for some $r$, let $k$ be the maximal $i$ such that $\lambda_i\geq r$. Then using $(\star)$ for this $k$, we get
		$$\sum_{\lambda_i\geq r}\lambda_i-r=\sum_{i=1}^{k}\lambda_i-r\leq\sum_{i=1}^{k}\mu_i-r\leq\sum_{\mu_i\geq r}\mu_i-r$$
		The result follows
	\end{itemize}
\end{proof}
\section{Fix points and closures}
\subsection{Neighborhoods of fixed points in type A}
Now we look into what the open subsets of Section 3.1 look like in type A. We denote by $<v_1,...v_n>_{t^{-1}}$ the $\mathbb{C}[t^{-1}]$-submodule of $\mathcal{K}^n$, spaned by $v_1$, $v_2$,...$v_n$\\
\begin{lem}
	In type A, with the usual description of $\mathcal{F}l$ we have
	$$U_1=\{tV_n\subset V_1\subset ...V_{n-1}\subset V_n|V_i\cap <t^{-1}e_1,...t^{-1}e_i,e_{i+1},...e_n>_{t^{-1}}=\{0\}\}$$
\end{lem} 
\begin{proof}
	Note that clearly the RHS contains 1. Further $<t^{-1}e_1,...t^{-1}e_i,e_{i+1},...e_n>_{t^{-1}}$ is $U_-$ stable, so $U_1\subset\{V_0\subset V_1\subset ...V_n\subset t^{-1}V_0|V_i\cap <t^{-1}e_1,...t^{-1}e_i,e_{i+1},...e_n>_{t^{-1}}=\{0\}\}$. Further any other fixed point does not satisfy the conditions, so as all $U_-$ orbits contain some fix point, we get the equality
\end{proof}
Hence we get from the lemma that the condition to check is intersections with $<t^{-1}e_1,...t^{-1}e_i,e_{i+1},...e_n>_{t^{-1}}$
\subsection{Computing open parts of components in type A}
Here we want to give an explicit description of the intersection $\mathcal{F}l\cap G(\mathcal{O})w\mathfrak{B}/\mathfrak{B}$, or more precisely we will compute the intersection with an open Schubert cell of the $G(\mathcal{O})$-orbit. We will actually use $\mathfrak{B}_-$-orbits, ie the Iwahori corresponding to the lower triangular matrices. This will change some of the computations by a multiplication by $w_0$, the longest element of $W_f$.\\
We will introduce some notation needed for this. For $w\in F$ and $1\leq i<j\leq n$ we introduce the integers $a_{ji} ^w$ as the smallest integer $k$ such that $w^{-1}(e_i-e_j)+k\delta$ is positive.\\
We then construct some lower triangular matrices $M^w$ that will describe the intersection, depending on $\sharp R^+$ variables denoted by $A_{ji}$ for $1\leq i<j\leq n$. We define the $ji$th entry as follows
$$M_{ji}^w=\sum_{i=i_1<i_2<...<i_k=j}c_{i_1i_2...i_k}\prod_{l=1}^{k-1} t^{a^w_{i_{l+1}i_l}}A_{i_{l+1}i_l}$$
Here the $c_{i_1i_2...i_k}$ are some constants depending on the regular semisimple $s$ given as $s=diag(s_i)$, the diagonal matrix with entries $s_i$. They are defined as follows
$$c_{i_1i_2...i_k}=\prod_{l=1}^{k-1}\frac{s_{i_l}-s_{i_{l+1}}}{s_{i_l}-s_{i_{k}}}$$
Here we interpret the case $k=1$ as an empty product and so the constant is $1$\\
We check that $M^w w\mathfrak{B}/\mathfrak{B}\subset \mathcal{F}l_\gamma$ and further it is of the correct dimension, so describes an open subset of the component we are looking at.\\
We do this in the following lemma
\begin{lem}
	The inverse of $M^w$ is given by the following
	$$(M^w)^{-1}_{ji}=\sum_{i=i_1<i_2<...<i_k=j}c'_{i_1i_2...i_k}\prod_{l=1}^{k-1} t^{a^w_{i_{l+1}i_l}}A_{i_{l+1}i_l}$$
	Here the constants $c'_{i_1i_2...i_k}$ are given by 
	$$c'_{i_1i_2...i_k}=(-1)^{k-1}\prod_{l=1}^{k-1}\frac{s_{i_l}-s_{i_{l+1}}}{s_{1}-s_{i_{l+1}}}$$
	Further this matrix satisfies $(M^w)^{-1}tsM^w\in {}^w\mathfrak{b}$
\end{lem}
\begin{proof}
	To check that the matrix with the above coefficients is the inverse it is enough to check
	$$\sum_{l=1}^{k} c_{i_1i_2...i_l}c'_{i_{l}i_2...i_k}=\begin{cases}
	1, & \text{if } k=1\\
	0, & \text{else}
	\end{cases}$$
	This is because this sum is exactly the coefficient of the monomial $\prod_{l=1}^{k-1}t^{a^w_{i_{l+1}i_l}}A_{i_{l+1}i_l}$ in the $i_ki_1$ entry of the product matrix.\\
	Note that the case $k=1$ is clear. Thus it remains to check
	$$\sum_{l=0}^{k} (-1)^l\prod_{r=1}^{l-1}\frac{s_{i_r}-s_{i_{r+1}}}{s_{r}-s_{i_{l}}}\prod_{r=l}^{k-1}\frac{s_{i_r}-s_{i_{r+1}}}{s_{l}-s_{i_{r+1}}}=0$$
	Or equivalently after clearing denominators and common factors
	$$\sum_{l=0}^{k} (-1)^l\prod_{\substack{ m<n\\m,n\neq l}}(s_{i_m}-s_{i_{n}})=0$$
	It is not hard to see that this is always divisible by $s_{i_m}-s_{i_{n}}$ for all $1\leq m<n\leq k$ by setting $s_{i_m}=s_{i_n}$ most summands vanish except the $m$ and $n$th factors. It is then easy to check that these summands are equal except for an opposite sign.\\
	Thus this polynomial is divisible by a degree ${k \choose 2}$ polynomial, but it is of degree ${k-1 \choose 2}$, so the polynomial has to be 0 if $k>1$ and for $k=1$, we clearly get $1$, ie this gives the identity matrix.\\
	Similarly to check what the matrix $(M^w)^{-1}tsM^w\in{}^w\mathfrak{b}$, it is enough to compute the following expression 
	$$\sum_{l=1}^{k} c_{i_1i_2...i_l}s_{i_l}c'_{i_{l}i_2...i_k}=0$$
	For $k>2$. That is because this gives exactly the coefficient of $\prod_{l=1}^{k-1}t^{a^w_{i_{l+1}i_l}}A_{i_{l+1}i_l}$in the $i_ki_1$ entry in $(M^w)^{-1}sM^w$. By the definition of $a^w_{ji}$ this exactly would mean that $(M^w)^{-1}tsM^w\in{}^w\mathfrak{b}$.\\
	We show this expression is $0$ unless $k\leq2$. To do this as above clearing denominators we reduce this to computing that
	$$\sum_{l=0}^{k} (-1)^ls_{i_l}\prod_{\substack{ m<n\\m,n\neq l}}(s_{i_m}-s_{i_{n}})=0$$
	But just as above this is always divisible by $s_{i_m}-s_{i_{n}}$ for all $1\leq m<n\leq k$
	and hence we get a ${k-1 \choose 2}+1$ degree polynomial being divisible by a ${k \choose 2}$ degree, so this yields the result for $k>2$\\ 
	Thus we get the entries of $(M^w)^{-1}tsM^w$ are given by 
	$$((M^w)^{-1}tsM^w)_{ji}=t^{a^w_{ji}+1}A_{ji}(s_j-s_i)$$
	And thus the result follows.
\end{proof}
\subsection{Determinants and fixed points}
We first note that we can reduce the computation of the set of fixed points of a component of $\mathcal{F}l_\gamma$ described by $F$ as above to computing a single element in the $W_f$-orbit.\\ To check this note that for any lift $\dot{w}$ of $w\in W_f$ $\dot{w}\mathcal{F}l_\gamma=\mathcal{F}l_{{}^w\gamma}$. Further as we've considered intersections with $G(\mathcal{O})$-orbits and $\dot{w}\in G$ we see the component $Y^s_x$ corresponding to $x\in F$ go to the corresponding component $Y^{wsw^{-1}}_x$ in the other affine Springer fiber. So if we check $y\in W$ is in $Y_x$ for every regular semisimple $s$, then $wy\in Y_x$ for any $w\in W_f$.\\
This means we can reduce the computation to checking that $\{y\leq x |y\in {}^fW\}$ are contained in $Y_x$ for $x\in F$.\\
Now we recall we need to compute the intersection of $M^xx\mathfrak{B}/\mathfrak{B}$ and $U_y$ is not empty, which by the description of $U_y$ amounts to checking $$M^xx<e_1,...e_i,te_{i+1},...te_n>\cap y<t^{-1}e_1,...t^{-1}e_i,e_{i+1},...e_n>_{t^{-1}}\neq {0}$$
For some matrix $M^x$ of the same form as given above for every $i$ and every element $y$ in $\{y\leq x |y\in {}^fW\}$.\\
In fact we will check a stronger condition. $F$ has a unique longest element, which we denote by $w_F$ and in fact $\forall x\in F$ $\exists z\in W$ such that $w_F=xz$ and $l(w_F)=l(x)+l(z)$. Then from this it follows that ${}^{w_F}\mathfrak{b}\cap\mathfrak{b}\subset{}^x\mathfrak{b}\cap\mathfrak{b}$ for $x\in F$ and thus $M^{w_F}x\mathfrak{B}/\mathfrak{B}\subset M^xx\mathfrak{B}/\mathfrak{B}$. We will check the condition that this subset intersects non-trivially with $U_y$ for $y\leq x$.\\
The condition for intersection can be rewritten in terms of certain determinants not vanishing. To describe this we first introduce a large rectangular matrix. To do this we introduce the notation $M_{ji}^{(k)}$ for the coefficient of $t^k$ in $M^{w_F}_{ji}$, ie $M_{ji}^{(k)}=\sum_{i=i_1<i_2<...<i_{j-1+1-k}=j}c_{i_1i_2...i_{j-i+1-k}}\prod_{l=1}^{j-i-k} A_{i_{l+1}i_l}$. With this notation the matrix is as follows\\
\\
\[\left[\begin{array}{c|c|c|c}
	\begin{matrix}
	1&0&\dots& 0\\
	M^{(0)}_{21}&1&\dots&0\\
	\vdots& &\ddots& \\
	& & &1\\
	\vdots&&&\vdots\\
	M_{n1}^{(0)}&\dots&\dots& M_{nj_0}
	\end{matrix}&0&\dots&0\\ \hline
	\begin{matrix}
	0&0&\dots& 0\\
	0&0&\dots&0\\
	M_{31}^{(1)}&0&\dots&0\\
	M_{41}^{(1)}&M_{42}^{(1)}&\dots&0\\
	\vdots& &\ddots& \\
	& & &M^{(1)}_{j_1+2j_1}\\
	\vdots&&&\vdots\\
	M_{n1}^{(1)}&\dots&\dots& M^{(1)}_{nj_1}
	\end{matrix}&\begin{matrix}
	1&\dots& 0\\
	M^{(0)}_{21}&\dots&0\\
	M^{(0)}_{31}&\ddots&0\\
	M^{(0)}_{41}&\ddots&0\\
	\vdots&\ddots& \\
	& &1\\
	\vdots&&\vdots\\
	M_{n1}^{(0)}&\dots& M^{(0)}_{nj_2}
	\end{matrix}&\dots&0\\ \hline
	\vdots&&\ddots&0\\ \hline
	\begin{matrix}
	0&\dots& 0\\
	\vdots&&\vdots\\
	M_{k+21}^{(k)}&\dots&0\\
	M_{k+31}^{(k)}&\dots&0\\
	\vdots& \ddots& \\
	&  &M^{(k)}_{j_1+k+1j_1}\\
	\vdots&&\vdots\\
	M_{n1}^{(k)}&\dots& M^{(k)}_{nj_1}
	\end{matrix}&
	\begin{matrix}
	0&\dots& 0\\
	\vdots&&\vdots\\
	M_{k+11}^{(k-1)}&\dots&0\\
	M_{k+21}^{(k-1)}&\dots&0\\
	\vdots& \ddots& \\
	&  &M^{(k-1)}_{j_2+kj_2}\\
	\vdots&&\vdots\\
	M_{n1}^{(k-1)}&\dots& M^{(k-1)}_{nj_2}
	\end{matrix}&\dots&
	\begin{matrix}
	1&\dots& 0\\
	M^{(0)}_{21}&\dots&0\\
	\vdots& \ddots& \\
	& &1\\
	\vdots&&\vdots\\
	M_{n1}^{(0)}&\dots& M_{nj_k}
	\end{matrix}
\end{array}\right]
\]\\
\\
Here this matrix depends on $x$ and $0\leq i<n$ via $j_l=\sharp\{x(\omega_i)_r\geq x(\omega_i)_1-l\}$ and $k=x(\omega_i)_1-x(\omega_i)_n-1$. Note that this $k$ is a positive number as by the combinatorial lemmas $x(\omega_i)$ is dominant. Here for $i=0$ we mean the similar conditions for $x(0)$.\\
Note that this matrix is just given by considering the vectors in\\ $M^xx<e_1,...e_i,te_{i+1},...te_n>$ in terms of the basis $t^re_j$, where the matrix is finite, because we use the condition that $x<e_1,...e_i,te_{i+1},...te_n>\supset t^{-x(\omega_i)_n}<e_1,...e_n>$\\
The determinant condition now consists of determinants of submatrices of the above, given by choosing the first $\sharp\{y(\omega_i)_j\geq x(\omega_i)_1-r\}$ rows of the $r$th block of rows. That is because these are exactly the entries whose corresponding basis element does not lie in $y<t^{-1}e_1,...t^{-1}e_i,e_{i+1},...e_n>_{t^{-1}}$ and hence to find an element in the intersection we need to find a non-trivial linear combination of the above vectors that has 0 entries corresponding to all these basis elements, or in other words we need to show that the square submatrix described is non-singular. Thus to show that the intersection above is trivial is equivalent to showing that these determinants do not vanish for some matrix $M^{w_F}$.\\
Note here we have marked some blocks that horizontally delimit the entries of a fixed valuation in a vector and vertically delimit the set of vectors with fixed minimal valuation.
We introduce some definitions
\begin{definition}
	\begin{enumerate}
		\item  We call the set of rows between 2 consecutive horizontal lines a fixed level sets
		\item  We call the set of columns between 2 consecutive vertical lines a set of fixed valuation vectors.
		\item  We call the entries delimited by 2 consecutive horizontal and 2 consecutive vertical lines a block of the matrix
		\item We call the diagonal block of the corresponding fixed valuation set as the row pivot and the corresponding diagonal block in the set of fixed valuation vectors as the column pivot
	\end{enumerate}
\end{definition} 
Here if we refer to the $k$th level set, we mean the $k$th fixed level starting from the top of the above matrix and similarly for the $k$th valuation vectors.
\subsection{Degree maximization algorithm}
We will check that the above determinants do not vanish identically on the open part of the component above described by finding an appropriately "high" degree monomial in some sense, whose coefficient is non-zero.\\
First we use the combinatorial lemmas above to note that the condition $y\leq x$ is equivalent to the condition
$$\sum_{j=r}^{x(\omega_l)_1}\sharp\{x(\omega_l)_i\geq j\}\leq \sum_{j=r}^{y(\omega_l)_1}\sharp\{y(\omega_l)_i\geq j\}\forall r,l$$
Note that this condition is equivalent to the property that the square matrix for which we are taking the determinant has more non-zero columns than rows, when looking at the first $r$ fixed level sets. This is another way of seeing that the condition of $y\leq x$ is necessary.\\
Now we will describe the monomial we're looking at. To do this we want to find a monomial that is only contributed by products of determinants coming from some square submatrices, one for each of the above fixed level sets. To see how to construct these submatrices and hence the monomial, we follow a greedy algorithm.\\
\begin{alg}
	Starting from the first block we counstruct the square submatrices by induction. We follow the algorithm as described 
\begin{enumerate}
\item  Once the square submatrix for the first $k-1$ fixed level set are chosen, construct the matrix of the $k$th level set as follows:\\
Take the first $min(\sharp\{x(\omega_i)_j\geq x(\omega_i)_1+1-k\},\sharp\{y(\omega_i)_j\geq x(\omega_i)_1+1-k\})$ columns corresponding to the $k$th valuation vectors, ie the number of columns of those fixed valuation vectors corresponding to the number of rows in the corresponding fixed level set, if there are fewer $k$ valuation vectors than rows in the corresponding fixed level set, and all of them otherwise.
\item If we still need to get more columns to have a square matrix we add all the columns not yet picked of the $k-1$st valuation vectors
\item If we still need to get more columns to have a square matrix proceed to the $k-2$nd and continue inductively.
\item We reach the $s$th valuation vectors for which we do not need to add all of them. For this set choose the first few columns necessary to complete the square submatrix for this fixed level set.
\item Continue constructing the next square submatrix
\end{enumerate}
\end{alg}
The monomial is now constructed as follows; each square determinant of the $k$th level set will contribute the following product of variables:
\begin{enumerate}
	\item For each column of the $k$th set of valuation vectors we just contribute 1
	\item For each column of the $k-1$st set of valuation vectors, if this is not the last set contributing columns, we denote by $c_k$ the largest integer $j$ such that $\lambda_j\geq\lambda_1-(k-1)$, by $r_k$ the largest integer $i$ such that $\mu_i\geq\lambda_1-(k-1)$ and by $m_1$ the number of columns of the $k-1$st set in the determinant of this block. Then this contributes to the monomial given by the product $\prod_{j=0}^{m_1-1}A_{c_k-jc_k-2-j}\prod_{i=c_{k-1}+m_1-j}^{c_k-3-j}A_{i+1i}$
	\item Similalry continuing inductively, for each column of the $k-s$th set, if this is not the last set contributing columns, we denote by $m_s$ the number of columns of the $k-s$th set in the determinant of this block, starting with the $l_s$. Then this contributes to the monomial the product of
	$\prod_{j=0}^{m_s-1} A_{c_k-\sum_{i=1}^{s-1}m_s-j{c_k-\sum_{i=1}^{s-1}m_i-s-1-j}} \prod_{i=l_s+m_s-j-1}^{c_k-\sum_{i=1}^{s-1}m_i-s-2-j} A_{i+1i}$ 
	\item Finally for each column of the $k-s$th set, if this is the last set contributing columns, we denote by $m_s$ the number of columns of the $k-s$th set in the determinant of this block and by $l_s$ the first column, Then this contributes to the monomial the product of
	 $\prod_{j=0}^{m_s-1}A_{l_s+s+j+1l_s+j}\prod_{i=l_s+s+j+1}^{m_0+j-1}A_{i+1i}$
\end{enumerate}
We check that the square determinants that we describe are the only ways to get these monomials.\\
To see this we start by considering maximizing degree for variables $A_{ij}$ where $i-j$ is large, ie we start by maximizing degree of $A_{n1}$, then of $A_{n2}$ and $A_{n-11}$ and so on in a lexicographic order. Note that these variables appear bellow some diagonal of blocks and further note that strictly below a diagonal they appear in entries of the matrix with monomials of strictly higher degree, so finding a monomial maximizing this will have to be given by these variables coming only from the corresponding block diagonal.\\
Now assume we have a square block lower triangular matrix (note the blocks, including the diagonal ones, do not need to be square and in fact can have $0$ rows). Further assume that there are as many blocks in the columns as in the rows (where we are counting empty blocks) and that the entries have a variable $x_i$ if the block is on the $ith$ subdiagonal under the diagonal of blocks. Ie we want a matrix that looks as follows
\[M=\left[
\begin{array}{c|c|c|c|c}
M_{11}&0&0&\dots &0\\ \hline
x_1M_{21}& M_{22}&0&\dots&0\\ \hline
x_2M_{31}& x_1M_{32}&M_{33}&\dots&0\\ \hline
\vdots& &&\ddots&0\\ \hline
x_{n-1}M_{n1}& x_{n-2}M_{n2}&x_{n-3}M_{n3}&\dots&M_{nn}
\end{array}\right]
\]
Here is $M_{ji}$ is an $r_j\times c_i$ for some $r_j$s and $c_i$s and further we have the condition that $\sum_{i=1}^{k}r_i\leq\sum_{i=1}^{k}c_i$, otherwise the determinant is always 0\\
We want to maximize the degree of $x_i$ in a dictionary order in computing the determinant of $M$, ie we first want to maximize the $x_{n-1}$ degree and among those with that maximal degree maximize the degree of $x_{n-2}$ and so on. To do this we show that a modified algorithm above gives all the ways of computing our monomial. To be precise we use the above algorithm except in Step $(1)$ and $(4)$ where we do not required to choose the first few columns, but any columns of that set. We refered to this as the modified algorithm.  We will then further check that the precise choice of monomial makes the condition of the first columns necessary.\\
Note that the modified algorithm only chooses all the columns of the first set of valuation vectors before reaching the last fixed level set if and only if $\sum_{i=1}^{k}r_i=\sum_{i=1}^{k}c_i$, ie if and only if at some valuation set you need to use all previous columns to get a square matrix. We break up the proof in the cases wheather this happens or not.
\begin{case}
Assume we have $\sum_{i=1}^{k}r_i<\sum_{i=1}^{k}c_i$ for $k\neq n$. Then the algorithm chooses an entry of $M_{n1}$ and hence uses a maximal degree element. Further it chooses it in the last step, so we can delete the corresponding column and row to get a similar matrix, still satisfying $\sum_{i=1}^{k}r_i\leq\sum_{i=1}^{k}c_i$. Applying the modified algorithm to this modified matrix is the same as looking at all the steps of the modified algorithm for $M$, without the last step. Thus by induction we can assume this gives a maximal degree element of the modified matrix and thus multiplying by $x_n$ we clearly get a maximal degree element of $M$
\end{case}
\begin{case}
	If we have $\sum_{i=1}^{k}r_i=\sum_{i=1}^{k}c_i$ for $k<n$ the matrix is block lower triangular (now with square blocks in the diagonal) of the form
	\[\begin{bmatrix}
	M_1&0\\
	N& M_2
	\end{bmatrix}
	\]
	with $M_i$ still broken up into squares with the correct degree polynomials as for $M$.\\
	The determinant breaks up as $det(M_1)det(M_2)$, so to maximize the degree, we just need to maximize it individually for $M_1$ and $M_2$ separately, but the modified algorithm, just runs first on $M_1$ and then on $M_2$ and hence by induction we can assume this gives us maximal degree elements
\end{case}
Now using the precise choice of monomial we will prove that we need to use the algorithm without modification. Thus we need to see that the construction of the monomial forces the algorithm to choose the first columns in Step $(1)$ and $(4)$\\

For Step $(1)$: Note that if $\sharp\{x(\omega_i)_j\geq x(\omega_i)_1+1-k\}\leq\sharp\{y(\omega_i)_j\geq x(\omega_i)_1+1-k\}$ we choose all the columns for this set of fixed valuation vectors, thus there is no difference with the modified algorithm. If $\sharp\{x(\omega_i)_j\geq x(\omega_i)_1+1-k\}\geq\sharp\{y(\omega_i)_j\geq x(\omega_i)_1+1-k\}$ the corresponding block will only have non-zero entries in the first $\sharp\{y(\omega_i)_j\geq x(\omega_i)_1+1-k\}$ columns. But the modified algorithm forces us to choose that many entries in this block, so we have to choose the first entries to get a non-zero contribution to the determinant.\\

For Step $(4)$: We proceed by induction on the subdiagonals of blocks. Ie we want to check that if the algorithm has to apply the modified Step $(4)$ at some block subdiagonal, if it was forced to apply the unmodified Step $(4)$ for all previous subdiagonals to get our monomial, then it will have to do so at this one too.\\
The start for the induction is the main diagonal, where it is equivalent to Step $(1)$, thus follows from the above. So assume for the first $s-1$ block subdiagonals the choices are given by the unmodifies algorithm. We proof this block by block.\\
Denote the block we're looking at by $B$ If we have to take all the remaining columns in this block, we are in the case where both algorithms agree. So assume we don't need to take all the remaining columns. Say we have to take $r$ columns in this block, starting by $l$, so the variables that should appear are $A_{s+ll}$,...$A_{s+l+r-1,l+r-1}$. We check that these variables chosen only appear in this block (given that for the previous subdiagonals we have used the unmodified algorithm).\\
First we proof this for a previous subdiagonal block. Denoted by $B'$ a previous block for which the algorithm gives us some choice of columns in the block. Then the column pivot for $B$, is further down than the row pivot of $B'$. To see this note that if this was not true, we would have to choose all the columns in the set of valuation vectors of $B$ in order to reach $B'$ using the above algorithm contradicting the fact that we have chosen columns fro $B$. Now denote by $r_{B'}$ the number of rows for the block $B'$. As the column pivot of $B$ comes after the row pivot of $B'$ we know the number of rows for the column pivot of $B$ is $\geq r_{B'}$. It then follows that $l>r_B$. It then immediately follows that the above variable can't appear in the previous blocks.\\
We now proof the result for subdiagonal blocks coming afterwards. Denote again $B'$ a block further down the subdiagonal for which the algorithm gives us some choice of column. Just as above the row pivot of $B$ appears before the column pivot of $B'$. But in order to choose from $B$ we must have chossen all the columns of the row pivot of $B$, but as $x\in F$ the number of columns in each block is strictly increasing. Now the above variables only appear in the first $l+r-1$ columns and further the column pivot of $B$ has at least $l+r-1$ columns, hence the row pivot of $B$ has strictly more and all are chosen, so for the column pivot of $B'$ we have to choose striclty more than $l+r-1$ columns, so as we choose the first of these by induction, we don't have any of the first $l+r-1$ columns to choose for $B'$. Thus in order to get the above monomials we need to follow the unmodified algorithm as required.\\

The above results forces us to choose the correct square submatrices for each fixed level set. Now we want to see that with the choice of monomial in fact it forces to choose a unique monomial in each entry and further that it breaks up further into smaller determinants, which we can compute.\\
For the blocks where we use Step $(1)$ there is a unique monomial in each entry. In the once where we apply Step $(4)$ by the above arguments we are forced to use the unique monomial including the variable as described above.\\
For the blocks where we apply the other steps note that we always choose $A_{ij}$ with $i-j$ maximal such that $i$ is maximal. Thus it forces us to choose those monomials with $A_{ij}$ in row $i$, as this is the maximal possible $i$, by proceeding by induction on subdiagonals.\\
This hence forces us to choose a particular monomial in each entry and further for the cases of Step $(2)$ and $(3)$ it forces us to choose the maximal available rows inductively in the subdiagonals.  
\subsection{Checking non-vanishing of determinants}
In this last section we compute the determinants that come out from the above algorithm and check they give a non-zero coefficient of our monomial, hence giving the required non-vanishing.\\
To do this recall that by the way we have constructed the determinants the $k-1$st columns uses up the last $m_1$ rows of the determinant corresponding to the $k$th fixed level set, the $k-2$nd the next $m_2$ and so on. Thus we can break the determinant further up into smaller determinants, each of which is given in a unique block for Step $(2)$ and $(3)$,  and one determinant given by the columns of Step $(1)$ and $(4)$ together.\\
The determinant of the square submatrices coming from steps $(2)$ and $(3)$ from a block in the $s$th subdiagonal will give us a factor of the coefficient given by the determinant of
\[\left[
\begin{BMAT}{ccc}{ccc}
\prod_{i=b}^{a-1}\frac{s_i-s_{i+1}}{s_i-s_{a+s}}&\dots&\prod_{i=b+m}^{a-1}\frac{s_i-s_{i+1}}{s_i-s_{a+s}}\\
\vdots&&\vdots\\
\prod_{i=b}^{a+m-1}\frac{s_i-s_{i+1}}{s_i-s_{a+m+s}}&\dots&\prod_{i=b+m}^{a+m-1}\frac{s_i-s_{i+1}}{s_i-s_{a+m+s}}
\end{BMAT}\right]\]
For some $a$, $b$ and $m$.\\
To see this we first note that for a block $B$ in the $s$th subdiagonal, as $x\in F$, the row pivot of $B$ has at least $s$ more columns than the column pivot. This means that the first non-zero entry in the last column of $B$ is at most 1 position down than the last column for the row pivot.\\
This shows that $\sum_{i=0}^{s} m_i<r_k$ for a block in the $s$th subdiagonal for which we apply steps $(2)$ and $(3)$, where $m_0$ is the number of columns in the row pivot and $r_k$ the number of rows of the $k$th fixed valuation set.\\
This shows that we can choose at every column at the $r$th row a non-zero monomial containing the factor $A_{rr-s}$.\\
Now we see that we can extract common non-zero factors from the above determinant, so we can reduce this to computing the determinant of 
\[\left[
\begin{BMAT}{cccc}{cccc}
1&s_{b}-s_{a+s}&\dots&\prod_{i=b}^{b+m-1}(s_i-s_{a+s})\\
1&s_{b}-s_{a+s+1}&\dots&\prod_{i=b}^{b+m-1}(s_i-s_{a+s+1})\\
\vdots&&&\vdots\\
1&s_{b}-s_{a+s+m}&\dots&\prod_{i=b}^{b+m-1}(s_i-s_{a+s+m})
\end{BMAT}\right]\]
But note that after taking appropriate linear combinations of the columns this reduces to the Vandermonde determinant, which is non-zero, as $s$ is regular semisimple.\\
Now we need to look at the case where we do not use all the remaining columns in a block. In this case we can use the columns corresponding to the row pivot, to eliminate the top elements until we are reduced to another block determinant.\\
To do this note we can just consider the entries with a factor of $A_{kl}$ for the $l$th column of the set and we look at a block in the $k-l$ subdiagonal. Then we can write the coefficients of the monomials including this variable as well as the column of the set starting with 1 as
\[\left[\begin{BMAT}{c}{ccccccc}
0\\
\vdots\\
0\\
1\\
\frac{s_l-s_k}{s_l-s_{k+1}}\\
\vdots\\
\frac{s_l-s_k}{s_l-s_{n}}\prod_{i=k}^{n-1}\frac{s_i-s_i+1}{s_i-s_{n}}
\end{BMAT} \right]\left[\begin{BMAT}{c}{ccccccc}
0\\
\vdots\\
0\\
1\\
\frac{s_k-s_{k+1}}{s_l-s_{k+1}}\\
\vdots\\
\prod_{i=k}^{n-1}\frac{s_i-s_i+1}{s_i-s_{n}}
\end{BMAT} \right]\]
Eliminating and clearing common non-zero factors we get the column
\[\left[\begin{BMAT}{c}{cccccccc}
0\\
\vdots\\
0\\
0\\
1\\
\frac{s_l-s_{k+1}}{s_l-s_{k+2}}\\
\vdots\\
\frac{s_l-s_{k+1}}{s_l-s_{n}}\prod_{i=k+1}^{n-1}\frac{s_i-s_i+1}{s_i-s_{n}}
\end{BMAT} \right]\]
Use now all the columns of the row pivot to eliminate as much as possible. Further we can eliminate non-zero common factors on rows. Then just as before the non-zero entries in a column start at most one position after the last column of the pivot, so eliminating always leaves behind a square matrix with coefficients of the form
\[\left[\begin{BMAT}{cccc}{cccc}
1&1&\dots&1\\
\frac{s_l-s_{k}}{s_l-s_{k+1}}&\frac{s_{l+1}-s_{k}}{s_{l+1}-s_{k+1}}&\dots& \frac{s_{l+m}-s_{k}}{s_{l+m}-s_{k+1}}\\
\vdots&&&\vdots\\
\frac{s_l-s_{k}}{s_l-s_{k+m}}&\frac{s_{l+1}-s_{k}}{s_{l+1}-s_{k+m}}&\dots& \frac{s_{l+m}-s_{k}}{s_{l+m}-s_{k+m}}
\end{BMAT}\right]\]
Using the last column to eliminate the first row we get the following matrix
\[\left[\begin{BMAT}{cccc}{cccc}
0&\dots&0&1\\
\frac{(s_l-s_{l+m})(s_k-s_{k+1})}{(s_l-s_{k+1})(s_{l+m}-s_{k+1})}&\dots&\frac{(s_{l+m-1}-s_{l+m})(s_k-s_{k+1})}{(s_l-s_{k+1})(s_{l+m}-s_{k+1})}& \frac{s_{l+m}-s_{k}}{s_{l+m}-s_{k+1}}\\
\vdots&&&\vdots\\
\frac{(s_l-s_{l+m})(s_k-s_{k+m})}{(s_l-s_{k+m})(s_{l+m}-s_{k+m})}&\dots&\frac{(s_l-s_{l+m})(s_k-s_{k+m})}{(s_l-s_{k+m})(s_{l+m}-s_{k+m})}& \frac{s_{l+m}-s_{k}}{s_{l+m}-s_{k+m}}
\end{BMAT}\right]\]
Note that this determinant can be reduced to the determinant of the smaller square block excluding the first row and the first column. After taking away non-zero factors we can reduce the determinant of this to computing the determinant of
\[\left[\begin{BMAT}{cccc}{cccc}
1&1&\dots&1\\
\frac{s_l-s_{k+1}}{s_l-s_{k+2}}&\frac{s_{l+1}-s_{k+1}}{s_{l+1}-s_{k+2}}&\dots& \frac{s_{l+m-1}-s_{k+1}}{s_{l+m-1}-s_{k+2}}\\
\vdots&&&\vdots\\
\frac{s_l-s_{k+1}}{s_l-s_{k+m}}&\frac{s_{l+1}-s_{k+1}}{s_{l+1}-s_{k+m}}&\dots& \frac{s_{l+m-1}-s_{k+1}}{s_{l+m-1}-s_{k+m}}
\end{BMAT}\right]\]
Note this is of the same shape as the previous determinant, hence by induction this will give a non-zero determinant as required.

\end{document}